\documentclass[12pt]{amsart}
\usepackage[all]{xy}
\usepackage{graphicx}
\usepackage{tikz}  

\usepackage{amsmath}
\usepackage{amssymb}
\usepackage{amsfonts}
\usepackage{mathrsfs}
\usepackage{mathtools}

\newcommand{\Cdb}{\ensuremath{\mathbb{C}}}

\newcommand{\Rdb}{\ensuremath{\mathbb{R}}}

\newcommand{\A}{\mbox{${\mathcal A}$}}
\newcommand{\B}{\mbox{${\mathcal B}$}}

\newcommand{\R}{\mbox{${\mathcal R}$}}
\renewcommand{\S}{\mbox{${\mathcal S}$}}

\newcommand{\norm}[1]{\Vert#1\Vert}
\newcommand{\bignorm}[1]{\bigl\Vert#1\bigr\Vert}
\newcommand{\Bignorm}[1]{\Bigl\Vert#1\Bigr\Vert}

\newtheorem{theorem}{Theorem}[section]
\newtheorem{lemma}[theorem]{Lemma}
\newtheorem{question}[theorem]{Question}
\newtheorem{corollary}[theorem]{Corollary}
\newtheorem{proposition}[theorem]{Proposition}

\theoremstyle{remark}
\newtheorem{remark}[theorem]{\bf Remark}
\theoremstyle{definition}

\numberwithin{equation}{section}

\begin{document}

\title[Extension of completely positive maps]
{On the extension of positive maps to Haagerup 
non-commutative $L^p$-spaces}

\bigskip
\author[C. Le Merdy]{Christian Le Merdy}
\email{clemerdy@univ-fcomte.fr}
\address{Laboratoire de Math\'ematiques de Besan\c con, 
Universit\'e de Franche-Comt\'e, 16 route de Gray
25030 Besan\c{c}on Cedex, FRANCE}

\author[S. Zadeh]{Safoura Zadeh}
\email{jsafoora@gmail.com}
\address{School of Mathematics, University of Bristol, United Kingdom}

\date{\today}


\begin{abstract} 
Let $M$ be a von Neumann algebra, let $\varphi$ be a normal faithful 
state on $M$ and let $L^p(M,\varphi)$ be the associated Haagerup 
non-commutative $L^p$-spaces, for $1\leq p\leq\infty$. Let 
$D\in L^1(M,\varphi)$ be the density of $\varphi$. Given a positive map $T\colon M\to M$ such that 
$\varphi\circ T\leq C_1\varphi$ 
for some $C_1\geq 0$, we study the boundedness of the $L^p$-extension
$T_{p,\theta}\colon D^{\frac{1-\theta}{p}}
MD^{\frac{\theta}{p}}\to L^p(M,\varphi)$ which maps 
$D^{\frac{1-\theta}{p}} x D^{\frac{\theta}{p}}$ to 
$D^{\frac{1-\theta}{p}} T(x) D^{\frac{\theta}{p}}$ for all 
$x\in M$. Haagerup-Junge-Xu showed that $T_{p,\frac12}$ is always bounded
and left open the question whether $T_{p,\theta}$ is bounded for $\theta\not=\frac12$. 
We show that for any $1\leq p<2$ and any $\theta\in [0,
2^{-1}(1-\sqrt{p-1})]\cup[2^{-1}(1+\sqrt{p-1}), 1]$, there exists 
a completely positive $T$ such that
$T_{p,\theta}$ is unbounded. We also show that if $T$ is
$2$-positive, then $T_{p,\theta}$ is bounded provided that 
$p\geq 2$ or $1\leq p<2$ and $\theta\in[1-p/2,p/2]$.
\end{abstract}

\maketitle

\noindent
{\it 2000 Mathematics Subject Classification :}  46L51.

\smallskip
\noindent
{\it Key words:} Non-commutative $L^p$-spaces, positive maps, extensions.

\vskip 1cm

\section{Introduction}\label{Intro}
Let $M$ be a von Neumann algebra equipped with a normal faithful 
state $\varphi$. Let $T\colon M\to M$ be a positive map such that 
$\varphi\circ T\leq C_1\varphi$ on the positive cone $M^+$, 
for some constant $C_1\geq 0$.
Assume first that $\varphi$ is a trace (that is, $\varphi(xy)=\varphi(yx)$ for all $x,y\in M$) and consider the associated 
non-commutative $L^p$-spaces ${\mathcal L}^p(M,\varphi)$  
(see e.g. \cite{DDDP, PX} or \cite[Chapter 4]{Hiai}). Let $C_\infty=\norm{T}$.
Then for all $1\leq p<\infty$, $T$ extends to a bounded map 
on ${\mathcal L}^p(M,\varphi)$, with
\begin{equation}\label{TracialExtension}
\bignorm{T\colon {\mathcal L}^p(M,\varphi)\longrightarrow
{\mathcal L}^p(M,\varphi)}\leq C_\infty^{1-\frac{1}{p}}C_1^{\frac{1}{p}},
\end{equation}
see \cite[Lemma 1.1]{JX}. This extension result plays a significant role 
in various aspects of operator theory on non-commutative $L^p$-spaces, 
in particular for the study of diffusion operators or semigroups on those spaces, 
see for example  \cite{A, DL, HRW} or \cite[Chapter 5]{JLX}.

Let us now drop the tracial assumption on $\varphi$. 
For any $1\leq p\leq \infty$, 
let $L^p(M,\varphi)$ denote the Haagerup non-commutative $L^p$-space
$L^p(M,\varphi)$ associated with $\varphi$ \cite{H, HJX, Hiai, Terp}. 
These spaces extend the tracial non-commutative $L^p$-spaces 
${\mathcal L}^p(\cdots)$
in a very beautiful way and many topics in operator theory which had 
been first studied on tracial non-commutative $L^p$-spaces were/are
investigated on Haagerup non-commutative $L^p$-spaces. This has led to 
several major advances, see in particular \cite{HJX}, \cite[Section 7]{JX},
\cite{CDS}, \cite{A2}  and \cite{JL}.

The question of extending a positive
map $T\colon M\to M$ to  $L^p(M,\varphi)$ was 
first considered in \cite[Section 7]{JX} and \cite[Section 5]{HJX}.
Let $D\in L^1(M,\varphi)$ be the density of $\varphi$, let
$1\leq p<\infty$ and let $\theta\in[0,1]$. Let 
$T_{p,\theta}\colon D^{\frac{1-\theta}{p}}
MD^{\frac{\theta}{p}}\to L^p(M,\varphi)$ be defined by 
\begin{equation}\label{T}
T_{p,\theta}\Bigl(D^{\frac{1-\theta}{p}} x D^{\frac{\theta}{p}}\Bigr)
= D^{\frac{1-\theta}{p}} T(x) D^{\frac{\theta}{p}},\qquad x\in M.
\end{equation}
(See Section 2 for the necessary background on $D$ and the above definition.)
Then \cite[Theorem 5.1]{HJX} shows that if 
$\varphi\circ T\leq C_1\varphi$,  then
$T_{p,\frac12}$ extends to a bounded map on $L^p(M,\varphi)$, 
with 
$$
\norm{T_{p,\frac12}\colon L^p(M,\varphi)\longrightarrow
L^p(M,\varphi)}\leq C_\infty^{1-\frac{1}{p}}C_1^{\frac{1}{p}}.
$$
This extends the tracial case (\ref{TracialExtension}), see 
Remark \ref{Tracial}. Furthermore, \cite[Proposition 5.5]{HJX} shows that  
if $T$ commutes with the modular automorphism group of $\varphi$, then 
$T_{p,\theta}=T_{p,\frac12}$ for all $\theta\in[0,1]$. 

In addition to the above results, Haagerup-Junge-Xu stated as 
an open problem the question 
whether $T_{p,\theta}$ is always bounded for $\theta\not=\frac12$
(see \cite[Section 5]{HJX}).
The main result of the present paper is a negative answer to this question.
More precisely, we show that if $1\leq p<2$ and if either $0\leq\theta<
2^{-1}(1-\sqrt{p-1})$ or $2^{-1}(1+\sqrt{p-1})<\theta\leq 1$, then there exists 
$M,\varphi$ as above and a unital completely positive map 
$T\colon M\to M$ such  that $\varphi\circ T=\varphi$ and $T_{p,\theta}$ is unbounded, see Theorem \ref{Main}.

We also show that for any $M,\varphi$ as above and for any
$2$-positive map $T\colon M\to M$ such that 
$\varphi\circ T\leq C_1\varphi$ for some $C_1\geq 0$,
then $T_{p,\theta}$ is bounded for all $p\geq 2$ and all
$\theta\in[0,1]$, see Theorem \ref{p=2}. In other words, the 
Haagerup-Junge-Xu problem has a positive solution for $p\geq 2$, provided that
we restrict to $2$-positive maps. We also show, under the 
same assumptions, that $T_{p,\theta}$ is bounded for all 
$1\leq p\leq 2$ and all $\theta\in[1-p/2,p/2]$, see 
Theorem \ref{GoodCase}.

Section 2 contains preliminaries on the $L^p(M,\varphi)$ 
and on the question
whether $T_{p,\theta}$ is bounded. Section 3 presents a way to compute
$\norm{T_{p,\theta}}$ in the case when $M=M_n$ is a matrix algebra,
which plays a key role in the last part of the paper.
Section 4 contains the extension results stated in the previous paragraph.
Finally, Sections 5 and 6 are devoted to the construction of 
examples for which $T_{p,\theta}$ is unbounded. 

\section{The extension problem}\label{Extension}
Throughout we consider a  von Neumann algebra $M$ and we 
let $M_*$ denote its predual. We let 
$M^+$ and $M_{*}^+$ denote the positive cones of $M$ and $M_*$,
respectively.

\smallskip
\subsection{Haagerup non-commutative $L^p$-spaces}\label{E1}

Assume that $M$ is $\sigma$-finite and
let $\varphi$ be a normal faithful state on $M$.
We shall briefly recall
the definition of the Haagerup non-commutative $L^p$-spaces
$L^p(M,\varphi)$ associated with $\varphi$, as well as some
of their main features. We refer the reader 
to \cite{H}, \cite[Section 1]{HJX}, \cite[Chapter 9]{Hiai}, \cite[Section 3]{PX} and \cite{Terp}
for details and 
complements. We note that $L^p(M,\varphi)$ can actually be defined when 
$\varphi$ is any normal faithful weight on $M$. The assumption that 
$\varphi$ is a state makes the description below a little simpler.

Let $(\sigma^\varphi_t)_{t\in{\mathbb R}}$ be the modular automorphism group
of $\varphi$ \cite[Chapter VIII]{T2} and let 
$$
\mathcal R = M \rtimes_{\sigma^\varphi} \Rdb\subset M\overline{\otimes}
B(L^2(\Rdb))
$$
be the resulting crossed product, see e.g. \cite[Chapter X]{T2}. 
If $M\subset B(H)$ for 
some Hilbert space $H$, then we have $\mathcal R \subset B(L^2(\Rdb;H))$.
Let us regard $M$ as a sub-von Neumann
algebra of $\R$ in the natural  way. 
Then $(\sigma^\varphi_t)_{t\in{\mathbb R}}$ is given by
\begin{equation}\label{lambda}
\sigma^\varphi_t(x) =\lambda(t)x\lambda(t)^*,\qquad t\in\Rdb,\ x\in M,
\end{equation}
where $\lambda(t)\in B(L^2(\Rdb;H))$ is defined 
by $[\lambda(t)\xi](s)=\xi(s - t)$ for all $\xi\in L^2(\Rdb;H)$.
This is a unitary.
For any $t\in\Rdb$, define 
$W(t) \in B(L^2(\Rdb;H))$ by $[W(t)\xi](s) =e^{-its}\xi(s)$ 
for all $\xi\in L^2(\Rdb;H)$.
Then the dual action 
$\widehat{\sigma}^\varphi\colon\Rdb
\to {\rm Aut}(\R)$ of $\sigma^\varphi$  
is defined by 
$$
\widehat{\sigma}^\varphi_t(x) = W(t)x W(t)^*,\qquad \ t\in\Rdb,\ x\in \R.
$$
(See \cite[§ VIII.2]{T2}.) A remarkable fact is that
for any $x\in\R$, $\widehat{\sigma}^\varphi_t(x)=x$ 
for all $t\in\Rdb$ if and only if 
$x\in M$.

There exists a unique normal semi-finite trace
$\tau_0$ on $\R$ such that 
$$
\tau_0\circ \widehat{\sigma}^\varphi_t =e^{-t}\tau_0,\qquad 
t\in\Rdb,
$$
see e.g. \cite[Theorem 8.15]{Hiai}.
This trace gives rise to the $*$-algebra $L^{0}(\R,\tau_0)$ of
$\tau_0$-measurable operators \cite[Chapter 4]{Hiai}. Then for any $1\leq p\leq\infty$, 
the Haagerup $L^p$-space $L^p(M,\varphi)$ is defined as
$$
L^p(M,\varphi) = \bigl\{
y\in L^0(\R,\tau_0)\, :\, \widehat{\sigma}^\varphi_t(y) =e^{-\frac{t}{p}}y\ \hbox{for all}\ t\in\Rdb\bigr\}.
$$
At this stage, this is just a $*$-subspace of $L^0(\R,\tau_0)$ (with no norm). 
One defines its positive cone as
$$
L^p(M,\varphi)^+= L^p(M,\varphi)\cap L^0(\R,\tau_0)^+.
$$
It follows from above that $L^\infty(M,\varphi)=M.$

Let $\psi\in M_{*}^+$, that we regard as a normal weight on $M$ and
let $\widehat{\psi}$ be its dual weight on $\R$ \cite[§ VIII.1]{T2}.
Let $h_\psi$ be the Radon-Nikodym derivative
of $\widehat{\psi}$ with respect to $\tau_0$. That is, $h_\psi$
is the unique positive operator affiliated 
with $\R$ such that 
$$
\widehat{\psi}(y) =\tau_0\Bigl(h_\psi^\frac12 yh_\psi^\frac12\Bigr),
\qquad y\in\R_+.
$$
It turns out that $h_\psi$ belongs to
$L^1(M,\varphi)^+$ for all
$\psi\in M_{*}^+$ and that the mapping
$\psi\mapsto h_\psi$ is a bijection
from $M_{*}^+$ onto $L^1(M,\varphi)^+$. 
This bijection readily extends to a linear isomorphism
$ M_* \longrightarrow L^1(M,\varphi)$, still denoted by 
$\psi\mapsto h_\psi$. Then $L^1(M,\varphi)$ is equipped with the norm 
$\norm{\,\cdotp}_1$ inherited from $M_*$, that is,
$\norm{h_\psi}_{1}=\norm{\psi}_{M_*}$ for all $\psi\in M_*$.
Next, for any $1\leq p<\infty$ and any $y\in L^p(M,\varphi)$, the positive 
operator $\vert y\vert$ belongs to $L^p(M,\varphi)$ as well (thanks to the 
polar decomposition) and hence 
$\vert y\vert^p$ belongs to $L^1(M,\varphi)$. This allows 
to define $\norm{y}_p=\norm{\vert y\vert^p}^{\frac{1}{p}}$ for all
$y\in L^p(M,\varphi)$. Then $\norm{\,\cdotp}_p$ is a complete norm on 
$L^p(M,\varphi)$.

The Banach spaces $L^p(M,\varphi)$, $1\leq p\leq\infty$, satisfy the 
following version of H\"older's inequality (see e.g. \cite[Proposition 9.17]{Hiai}).

\begin{lemma}\label{Holder} Let $1\leq p,q,r\leq \infty$ 
such that $p^{-1}+q^{-1}
=r^{-1}$. Then for all $x\in L^p(M,\varphi)$ and all
$y\in L^q(M,\varphi)$, the product $xy$ belongs to $L^r(M,\varphi)$ and $\norm{xy}_r\leq
\norm{x}_p\norm{y}_q$.
\end{lemma}

Let $D$ be the Radon-Nikodym derivative
of $\widehat{\varphi}$ with respect to $\tau_0$
and recall that $D\in L^1(M,\varphi)^+$. This operator is called the density of $\varphi$. 
Recall that we regard $M$ as a sub-von Neumann
algebra of $\R$. Then $D^{it}=\lambda(t)$ is a unitary of $\R$ for all $t\in\Rdb$ 
and 
\begin{equation}\label{MAG}
\sigma^\varphi_t(x) = D^{it}xD^{-it},\qquad t\in\Rdb,\ x\in M.
\end{equation}

Let ${\rm Tr}\colon L^1(M,\varphi)\to \Cdb$ be defined by
${\rm Tr}(h_\psi)=\psi(1)$ for all $\psi\in M_*$.
This functional has two remarkable properties. First, 
for all $x\in M$ and all $\psi\in M_*$, we have
\begin{equation}\label{duality}
{\rm Tr}(h_\psi x) = \psi(x).
\end{equation}
Second if $1\leq p,q\leq \infty$ are such that $p^{-1}+q^{-1}
=1$, then for all $x\in L^p(M,\varphi)$ and all
$y\in L^q(M,\varphi)$, we have
$$
{\rm Tr}(xy) = {\rm Tr}(yx).
$$
This tracial property 
will be used without any further comment in the paper.

It follows from the definition of $\norm{\,\cdotp}_1$ 
and (\ref{duality}) that the duality pairing 
$\langle x,y\rangle ={\rm Tr}(xy)$ for $x\in M$ and $y\in L^1(M,\varphi)$
yields an isometric isomorphism
\begin{equation}\label{1Dual-L1}
L^1(M,\varphi)^*\simeq M.
\end{equation}

As a special case of (\ref{duality}), we have
\begin{equation}\label{1Dual-L1-1}
\varphi(x)={\rm Tr}(Dx),\qquad x\in M.
\end{equation}
We note that $L^2(M,\varphi)$ is a  space for the inner product
$(x\vert y)={\rm Tr}(y^*x)$. Moreover by 
(\ref{1Dual-L1-1}), we have 
\begin{equation}\label{L2}
\varphi(x^*x)=\norm{xD^\frac12}_2^2\quad\hbox{and}\quad 
\varphi(xx^*)=\norm{D^\frac12 x}_2^2,\qquad x\in M.
\end{equation}

We finally mention a useful tool.
Let $M_a\subset M$ be the subset of
all $x\in M$ such that $t\mapsto \sigma^\varphi_t(x)$ extends to an entire
function $z\in\Cdb\mapsto \sigma_z^\varphi(x) \in M$. (Such
elements are called analytic).
It is well-known that $M_a$ is a $w^*$-dense $*$-sub-algebra of $M$ \cite[Section VIII.2]{T2}. Furthermore, 
\begin{equation}\label{MAG2}
\sigma_{i\theta}(x)=D^{-\theta}xD^{\theta},
\end{equation}
for all $x\in M_a$ and all
$\theta\in[0,1]$, 
and $M_aD^{\frac{1}{p}} = D^{\frac{1}{p}}M_a$ is dense in 
$L^p(M,\varphi)$, for all $1\leq p<\infty$. See 
\cite[Lemma 1.1]{JX1} and its proof
for these properties.

\smallskip
\subsection{Extension of maps $M\to M$}\label{E2}
Given any linear map $T\colon M\to M$,
we say that $T$ is positive if
$T(M^+)\subset M^+$. This implies that $T$ is bounded.
For any $n\geq 1$, we say that $T$ is 
$n$-positive if the tensor extension  map $I_{M_n}\otimes T\colon M_n\overline{\otimes} M
\to M_n\overline{\otimes}M$ is positive. (Here $M_n$ is the algebra
of $n\times n$ matrices.) Next, we say that $T$ is completely positive if
$T$ is $n$-positive for all $n\geq 1$. See e.g. \cite{Pa} for basics on these notions.

Consider any $\theta\in[0,1]$ and $1\leq p<\infty$. 
It follows from Lemma 
\ref{Holder} that $D^{\frac{1-\theta}{p}} x 
D^{\frac{\theta}{p}}$ belongs to $L^p(M,\varphi)$ for all $x\in M$.
We set
\begin{equation}\label{A}
\A_{p,\theta} = D^{\frac{(1-\theta)}{p}} M D^{\frac{\theta}{p}}
\subset L^p(M,\varphi).
\end{equation}
It turns out that this is a dense subspace, see \cite[Lemma 1.1]{JX1}.

Let $T\colon M\to M$ be any bounded linear map.
For any $(p,\theta)$ as above, 
define a linear map $T_{p,\theta}\colon \A_{p,\theta}\to \A_{p,\theta}$ by
(\ref{T}).
The question we consider in this paper is whether $T_{p,\theta}$ 
extends to a bounded map $L^p(M,\varphi)\to L^p(M,\varphi)$ 
in the case when $T$ is
$2$-positive and $\varphi\circ T\leq \varphi$
on $M_+$. More precisely, we consider the following:

\begin{question}\label{Q}
Determine the pairs 
$(p,\theta)\in[1,\infty)\times[0,1]$ such that 
$$
T_{p,\theta}\colon L^p(M,\varphi)\longrightarrow L^p(M,\varphi)
$$
is bounded for all $(M,\varphi)$ as above and all
$2$-positive maps $T\colon M\to M$ satisfying
$\varphi\circ T\leq \varphi$ on $M_+$.
\end{question}

As in the introduction, we could consider maps such that
$\varphi\circ T\leq C_1\varphi$ for some $C_1\geq 0$. However 
by an obvious scaling, there is no loss in considering $C_1=1$ only.

\begin{remark}\label{Known}
Question \ref{Q} originates from the Haagerup-Junge-Xu paper
\cite{HJX}. In Section 5 of the latter paper, the authors 
consider two von Neumann algebras $M,N$, and normal faithful states
$\varphi\in M_*$ and $\psi\in N_*$ with respective densities 
$D_\varphi\in L^1(M,\varphi)$ and $D_\psi\in L^1(N,\psi)$. Then they consider 
a positive map
$T\colon M\to N$
such that $\psi\circ T\leq C_1\varphi$ for some
$C_1>0$. Given any $(p,\theta)\in[1,\infty)\times[0,1]$, they
define $T_{p,\theta}\colon D_\varphi^{\frac{1-\theta}{p}} M D_\varphi^{\frac{\theta}{p}}\to L^p(N,\psi)$ by
$$
T_{p,\theta}\Bigl(D_\varphi^{\frac{1-\theta}{p}} x 
D_\varphi^{\frac{\theta}{p}}\Bigr) = D_\psi^{\frac{1-\theta}{p}} T(x) D_\psi^{\frac{\theta}{p}},\qquad x\in M.
$$
In \cite[Theorem 5.1]{HJX},
they show that $T_{p,\frac12}$ is bounded and that setting 
$C_\infty=\norm{T}$, we have
$\norm{T_{p,\frac12}\colon L^p(M,\varphi)\to L^p(N,\psi)}
\leq C_{\infty}^{1-\frac{1}{p}} C_{1}^{\frac{1}{p}}$.
Then after the statement of \cite[Proposition 5.4]{HJX}, they mention that 
the boundedness of $T_{p,\theta}$ for $\theta\not=\frac12$ is an open
question.
\end{remark}

\begin{remark}\label{CondExp}
We wish to point out a special case which will be used in Section \ref{TP}.
Let $B$ be a von Neumman algebra equipped with a normal 
faithful state $\psi$. Let $A\subset B$ be a sub-von Neumann algebra
which is stable under the modular automorphism group of $\psi$ (i.e.
$\sigma^\psi_t(A)\subset A$ for all $t\in\Rdb$). Let $\varphi=\psi_{\vert A}$
be the restriction of $\psi$ to $A$. Let $D\in L^1(A,\varphi)$ and
$\Delta\in L^1(B,\psi)$ be the densities of $\varphi$ and
$\psi$, respectively. On the one hand, 
it follows from \cite[Theorem 5.1]{HJX} (see Remark \ref{Known}) that there exists, for every $1\leq p<\infty$, a contraction
$$
\Lambda(p)\colon L^p(A,\varphi)\longrightarrow L^p(B,\psi)
$$
such that $[\Lambda(p)](D^{\frac{1}{2p}}xD^{\frac{1}{2p}}) =
\Delta^{\frac{1}{2p}}x \Delta^{\frac{1}{2p}}$ for all $x\in A$. 

On the other hand,  there exists a 
unique normal conditional expectation $E\colon B\to A$ 
such that $\psi=\varphi\circ E$ on $B$, by  \cite[Theorem IX.4.2]{T2}.
Moreover it is easy to check 
that under the natural identifications $L^1(A,\varphi)^*\simeq A$ and 
$L^1(B,\psi)^*\simeq B$, see (\ref{1Dual-L1})
and the discussion preceding it, we have
$$
\Lambda(1)^*=E. 
$$

Now using \cite[Theorem 5.1]{HJX} again, there exists, 
for every $1\leq p<\infty$, a contraction
$E(p)\colon L^p(B,\psi)\to L^p(A,\varphi)$ such that 
$[E(p)](\Delta^{\frac{1}{2p}}y \Delta^{\frac{1}{2p}}) =
D^{\frac{1}{2p}}E(y)D^{\frac{1}{2p}}$ for all $y\in B$. 
It is clear that $E(p)\circ \Lambda(p)=I_{L^p(A,\varphi)}$.
Consequently, $\Lambda(p)$ is an isometry.

We refer to \cite[Section 2]{JX1} for more on this.
\end{remark}

\begin{remark}\label{Tracial}
Let $T\colon M\to M$ be a positive map and let $\varphi,D$ as 
in Subsection \ref{E1}.
Assume that $\varphi$ is tracial and for any $1\leq p<\infty$,
let ${\mathcal L}^p(M,\varphi)$
be the (classical) non-commutative $L^p$-space with respect to the trace
$\varphi$ \cite[Section 4.3]{Hiai}. That is,
${\mathcal L}^p(M,\varphi)$ is the completion of $M$ for the norm
$$
\norm{x}_{{\mathcal L}^p(M,\varphi)} = \bigl(\varphi(\vert x\vert^p)\bigr)^{\frac{1}{p}},\qquad x\in M.
$$
In this case, $D$ commutes with $M$ and 
$$
\norm{D^{\frac{1}{p}} x}_{L^p(M,\varphi)} = 
\norm{x}_{{\mathcal L}^p(M,\varphi)},\qquad x\in M,
$$
see e.g. \cite[Example 9.11]{Hiai}. Hence, 
$T_{p,\theta}=T_{p,0}$ for all $1\leq p<\infty$ and all $\theta\in [0,1]$ and moreover,  
$T_{p,0}$ is bounded if and only if $T$ extends to a bounded map
${\mathcal L}^p(M,\varphi)\to {\mathcal L}^p(M,\varphi)$. Thus, in the 
tracial case,
the fact that $T_{p,0}$ is bounded under 
the assumption $\varphi\circ T\leq C_1\varphi$ 
is equivalent to the result mentionned in the first paragraph of
Section \ref{Intro}, see (\ref{TracialExtension}).
\end{remark}

\section{Computing $\norm{T_{p,\theta}}$ 
on semifinite von Neumann algebras}\label{SF}
As in the previous section, we let $M$ be a von Neumann algebra
equipped with a normal faithful state $\varphi$ and we let 
$D\in L^1(M,\varphi)^+$ be the density of $\varphi$. We assume further that
$M$ is semifinite and we let $\tau$ be a distinguished normal
semifinite faithful trace on $M$. For any $1\leq p\leq \infty$, we let 
${\mathcal L}^p(M,\tau)$ be the non-commutative $L^p$-space 
with respect to $\tau$. Although ${\mathcal L}^p(M,\tau)$ is isometrically 
isomorphic to the Haagerup
$L^p$-space $L^p(M,\tau)$,
it is necessary for our purpose to consider ${\mathcal L}^p(M,\tau)$
as such.

Let us give a brief account, for which we refer e.g. 
to \cite[Section 4.3]{Hiai}.
Let ${\mathcal L}^0(M,\tau)$ be the $*$-algebra of all $\tau$-measurable 
operators on $M$. 
For any $p<\infty$, ${\mathcal L}^p(M,\tau)$ is the Banach space
of all $x\in {\mathcal L}^0(M,\tau)$ such that $\tau(\vert x\vert^p)<\infty$, 
equipped with the norm
$$
\norm{x}_{{\mathcal L}^p(M,\tau)}=\bigl(\tau(\vert x\vert^p)\bigr)^{\frac{1}{p}},
\qquad x\in {\mathcal L}^p(M,\tau).
$$
Moreover ${\mathcal L}^\infty(M,\tau)=M$. The following analogue of
Lemma \ref{Holder} holds true: whenever
$1\leq p,q,r\leq\infty$ are such that $p^{-1}+q^{-1}=r^{-1}$,
then  for all 
$x\in {\mathcal L}^p(M,\tau)$ and $y\in {\mathcal L}^q(M,\tau)$,
$xy$ belongs to ${\mathcal L}^r(M,\tau)$, with
$\norm{xy}_r\leq\norm{x}_p\norm{x}_q$ (H\"older's inequality). 
Furthermore, we have 
an isometric identification
\begin{equation}\label{Dual-L1}
{\mathcal L}^1(M,\tau)^*\simeq M
\end{equation}
for the duality pairing given by $\langle x,y\rangle=\tau(yx)$ for all
$x\in M$ and $y\in {\mathcal L}^1(M,\tau)$.

Let $\gamma\in {\mathcal L}^1(M,\tau)$ be associated
with $\varphi$ in the identification (\ref{Dual-L1}), that is,
\begin{equation}\label{Dual-L1-1}
\varphi(x)=\tau(\gamma x),\qquad x\in M.
\end{equation}
Then $\gamma$ is positive and it is clear from H\"older's inequality that for any $1\leq p<\infty$,
$\theta\in[0,1]$ and $x\in M$, the product 
$\gamma^{\frac{1-\theta}{p}} x 
\gamma^{\frac{\theta}{p}}$ belongs to 
${\mathcal L}^p(M,\tau)$.

It is well-known that ${\mathcal L}^p(M,\tau)$ and $L^p(M,\varphi)$
are isometrically isomorphic (apply Remark 9.10 and Example 9.11 
in \cite{Hiai}). The following lemma
provides concrete isometric isomorphisms between these two spaces.

\begin{lemma}\label{D-Gamma}
Let $1\leq p<\infty$ and $\theta\in[0,1]$. Then for all $x\in M$, we have
$$
\bignorm{\gamma^{\frac{1-\theta}{p}} x 
\gamma^{\frac{\theta}{p}}}_{{\mathcal L}^p(M,\tau)} = 
\bignorm{D^{\frac{1-\theta}{p}} x 
D^{\frac{\theta}{p}}}_{L^p(M,\varphi)}.
$$
\end{lemma}

Before giving the proof of this lemma, we recall a classical tool.
For any $\theta\in[0,1]$, define an
embedding $J_\theta\colon M\to 
L^1(M,\varphi)$ by letting 
$$
J_\theta(x)=D^{1-\theta}xD^\theta,\qquad x\in M.
$$
Consider $(J_\theta(M), L^1(M,\varphi))$ 
as an interpolation couple, the norm on $J_\theta(M)$ being given by the 
norm on $M$, that is, 
\begin{equation}\label{NormM1}
\bignorm{D^{1-\theta}xD^\theta}_{J_\theta(M)} = \norm{x}_M,\qquad x\in M.
\end{equation}
For any $1\leq p\leq\infty$,
let 
\begin{equation}\label{Cp}
C(p,\theta)=\bigl[J_\theta(M), L^1(M,\varphi)\bigr]_{\frac{1}{p}}
\end{equation}
be the resulting interpolation space 
provided by the complex interpolation
method \cite[Chapter 4]{BL}. Regard $C(p,\theta)$ as a
subspace of $L^1(M,\varphi)$ in the natural way. 
Then Kosaki's theorem \cite[Theorem 9.1]{Ko} (see also \cite[Theorem 9.36]{Hiai}) asserts that 
$C(p,\theta)$ is equal to $D^{\frac{1-\theta}{p'}}L^{p}(M,\varphi)D^{\frac{\theta}{p'}}$ and that
\begin{equation}\label{Ko1}
\bignorm{D^{\frac{1-\theta}{p'}}y
D^{\frac{\theta}{p'}}}_{C(p,\theta)}
=\norm{y}_{L^p(M,\varphi)},\qquad y\in L^p(M,\varphi).
\end{equation}
Here $p'$ is the conjugate index of $p$, so that $D^{\frac{1-\theta}{p'}}y
D^{\frac{\theta}{p'}}$ belongs to $L^1(M,\varphi)$ provided that $y$ belongs to $L^p(M,\varphi)$.

Likewise, let $j_\theta\colon M\to 
{\mathcal L}^1(M,\tau)$ be defined
by $j_\theta(x)=\gamma^{1-\theta}x\gamma^\theta$ for all $x\in M$. 
Consider $(j_\theta(M), {\mathcal L}^1(M,\tau))$ as an 
interpolation couple, the norm on $j_\theta(M)$ being given by the 
norm on $M$, and set
\begin{equation}\label{cp}
c(p,\theta)=[j_\theta(M), {\mathcal L}^1(M,\tau)]_{\frac{1}{p}},
\end{equation}
regarded as a
subspace of ${\mathcal L}^1(M,\tau)$. 
Then arguing as in the proof of \cite[Theorem 9.1]{Ko}, one obtains that 
$c(p,\theta)$ is equal to $\gamma^{\frac{1-\theta}{p'}}
{\mathcal L}^{p}(M,\tau)\gamma^{\frac{\theta}{p'}}$ 
and that
\begin{equation}\label{Ko2}
\bignorm{\gamma^{\frac{1-\theta}{p'}}y
\gamma^{\frac{\theta}{p'}}}_{c(p,\theta)}
=\norm{y}_{{\mathcal L}^{p}(M,\tau)},\qquad y\in {\mathcal L}^{p}(M,\tau).
\end{equation}

\smallskip

\begin{proof}[Proof of Lemma \ref{D-Gamma}.] 
We fix some $\theta\in[0,1]$.
We start with the case $p=1$. Let $x\in M$. For any $x'\in M$, we have 
$\tau(\gamma xx')={\rm Tr}(Dxx')$ and hence 
$\vert \tau(\gamma xx')\vert =\vert{\rm Tr}(Dxx')\vert$, by (\ref{1Dual-L1-1}) and 
(\ref{Dual-L1-1}). Taking the supremum 
over all $x'\in M$ with $\norm{x'}_M\leq 1$, it therefore follows 
from 
(\ref{1Dual-L1}) and (\ref{Dual-L1}) that 
\begin{equation}\label{Prelim}
\bignorm{\gamma x}_{{\mathcal L}^1(M,\tau)} = 
\bignorm{Dx}_{L^1(M,\varphi)},\qquad x\in M.
\end{equation}
Now assume that  $x\in M_a$ (the space of analytic elements of $M$). According to (\ref{MAG2}), we have
$D\sigma^\varphi_{i\theta}(x)=D^{1-\theta}xD^{\theta}$.  
Likewise, $\sigma^\varphi_t(x)=\gamma^{it}x\gamma^{-it}$
for all $t\in \Rdb$, by \cite[Theorem VIII.2.11]{T2}, hence $\sigma^\varphi_{i\theta}(x)=\gamma^{-\theta}x\gamma^\theta$. Hence 
we have $\gamma\sigma^\varphi_{i\theta}(x)=\gamma^{1-\theta}x\gamma^{\theta}$.
Applying (\ref{Prelim}) with 
$\sigma^\varphi_{i\theta}(x)$ 
in place of $x$, we deduce that
\begin{equation}\label{p1}
\bignorm{\gamma^{(1-\theta)} x 
\gamma^{\theta}}_{{\mathcal L}^1(M,\tau)} = 
\bignorm{D^{(1-\theta)} x 
D^{\theta}}_{L^1(M,\varphi)}.
\end{equation}
Consider 
the standard representation $M\hookrightarrow B(L^2(M,\varphi))$
and consider an arbitrary $x\in M$. Assume that $\theta\geq\frac12$.
There exists a net
$(x_i)_i$ in $M_a$ such that $x_i\to x$ strongly. Then
$x_iD^\frac12\to xD^\frac12$ in $L^2(M,\varphi)$. Applying Lemma \ref{Holder} (H\"older's inequality), we deduce that
$D^{1-\theta}x_iD^{\theta}=D^{1-\theta}(x_iD^\frac12)D^{\theta-\frac12}$ converges to 
$D^{1-\theta} xD^{\theta}$ in $L^{1}(M,\varphi)$. (This result can also be formally deduced from \cite[Lemma 2.3]{J}.)
Likewise,  
$\gamma^{1-\theta}x_i\gamma^{\theta}$ converges to 
$\gamma^{1-\theta} x\gamma^{\theta}$ in ${\mathcal L}^{1}(M,\tau)$. Consequently,
(\ref{p1}) holds true for $x$. Changing $x$ into $x^*$, we obtain this result 
as well if  $\theta<\frac12$.
This proves the result when $p=1$.

We further note that the proof that 
$\A_{1,\theta}=D^{(1-\theta)} M 
D^{\theta}$ is dense in $L^1(M,\varphi)$ shows as well that the space 
$\gamma^{1-\theta} M 
\gamma^{\theta}$ is dense in ${\mathcal L}^1(M,\tau)$. Thus,
(\ref{p1}) provides an isometric isomorphism
$$
\Phi\colon L^1(M,\varphi)\longrightarrow {\mathcal L}^1(M,\tau)
$$
such that 
$$
\Phi\bigl(D^{1-\theta} x 
D^{\theta}\bigr) = \gamma^{1-\theta} x \gamma^{\theta},\qquad x\in M.
$$

Now let $p>1$ and consider the interpolation spaces
$C(p,\theta)$ and $c(p,\theta)$ defined by (\ref{Cp}) and (\ref{cp}). Since 
$j_\theta=\Phi\circ J_\theta$, the mapping
$\Phi$ restricts to an isometric isomorphism from $C(p,\theta)$ onto $c(p,\theta)$.
Let $x\in M$. Applying (\ref{Ko2}) and
(\ref{Ko1}), we deduce that
\begin{align*}
\bignorm{\gamma^{\frac{1-\theta}{p}} x 
\gamma^{\frac{\theta}{p}}}_{{\mathcal L}^p(M,\tau)} & =
\bignorm{\gamma^{1-\theta}x\gamma^\theta}_{c(p,\theta)}\\
& =
\bignorm{D^{1-\theta}x D^\theta}_{C(p,\theta)}\\
&= \bignorm{D^{\frac{1-\theta}{p}} x 
D^{\frac{\theta}{p}}}_{L^p(M,\varphi)},
\end{align*}
which proves the result.
\end{proof}

The following is a straightforward consequence 
of Lemma \ref{D-Gamma}. Given any $T\colon M\to M$, it
provides a concrete way to compute the norm of the operator $T_{p,\theta}$ associated with $\varphi$.
Note that in this statement, this norm may be infinite.

\begin{corollary}
Let $1\leq p<\infty$, let $\theta\in[0,1]$ and let
$T\colon M\to M$ be any bounded map. Then 
$$
\norm{T_{p,\theta}} = \sup\Bigl\{
\bignorm{\gamma^{\frac{1-\theta}{p}}T(x)
\gamma^{\frac{\theta}{p}}}_{p}\, :\,
x\in M,\ \bignorm{\gamma^{\frac{1-\theta}{p}} x 
\gamma^{\frac{\theta}{p}}}_{p}\leq 1\Bigr\}.
$$
\end{corollary}

Let $n\geq 1$ be an integer and consider the special 
case when $M=M_n$, equipped with its usual 
trace ${\rm tr}$. For any $\varphi$
and $T\colon M_n\to M_n$ as above,
$T_{p,\theta}$ is trivially bounded for all
$1\leq p<\infty$ and $\theta$ since 
$L^p(M_n,\varphi)$ is finite dimensional. However 
we will see in Sections \ref{TP} and \ref{No} that
finding (lower) estimates of the norm 
of $T_{p,\theta}$ in this setting will be
instrumental to devise 
counter-examples on infinite dimensional von Neumann algebras. This is why we give a version of the preceding corollary in this specific case.

For any $1\leq p<\infty$, 
let $S^p_n={\mathcal L}^p(M_n,{\rm tr})$ denote the
$p$-Schatten class over $M_n$.

\begin{proposition}\label{Transfer}
Let $\Gamma\in M_n$ be a positive definite matrix such that ${\rm tr}(\Gamma)=1$ and let $\varphi$ be the faithful state 
on $M_n$ associated with $\Gamma$, that is,
$\varphi(X)={\rm tr}(\Gamma X)$
for all $X\in M_n$.
Let  $T\colon M_n\to M_n$ be any linear map. For any $p\in[1,\infty)$ and $\theta\in[0,1]$, let 
$U_{p,\theta}\colon S^p_n\to S^p_n$ be defined by
\begin{equation}\label{Up}
U_{p,\theta}(Y) = 
\Gamma^{\frac{1-\theta}{p}} T\bigl(\Gamma^{-\frac{1-\theta}{p}}
Y\Gamma^{-\frac{\theta}{p}}\bigr)\Gamma^{\frac{\theta}{p}},
\qquad Y\in S^p_n.
\end{equation}
Then 
$$
\bignorm{T_{p,\theta}\colon L^p(M_n,\varphi)\longrightarrow L^p(M_n,\varphi)} =
\bignorm{U_{p,\theta}\colon S^p_n\longrightarrow S^p_n}.
$$
\end{proposition}

\section{Extension results}\label{Yes} 
This section is devoted to two cases
for which Question \ref{Q} 
has a positive answer.
Let $M$ be a von Neumann algebra equipped with a faithful normal state $\varphi$ and
let $D\in L^1(M,\varphi)^+$ denote its density.

\begin{theorem}\label{p=2} 
Let $T\colon M\to M$ be a $2$-positive map
such that $\varphi\circ T\leq \varphi$. For any $p\geq 2$ and for any 
$\theta\in[0,1]$, the mapping $T_{p,\theta}\colon \A_{p,\theta}\to 
\A_{p,\theta}$ defined by (\ref{T}) extends to a 
bounded map $L^p(M,\varphi)\to L^p(M,\varphi)$.
\end{theorem}

\begin{proof}
Consider a $2$-positive map $T\colon M\to M$ such 
that $\varphi\circ T\leq \varphi$.
We start with the case $p=2$. 
For any
$x\in M$, we have
$$
T(x)^*T(x)\leq \norm{T} T(x^*x),
$$
by the Kadison-Schwarz inequality \cite{Choi}. By (\ref{L2}), we have
$$
\norm{T(x)D^\frac12}_2^2= 
\varphi\bigl(T(x)^*T(x)\bigr)\leq \norm{T}\varphi\bigl(T(x^*x)\bigr)
\leq \norm{T}\varphi(x^*x)=\norm{T}\norm{xD^\frac12}_2^2.
$$
This shows that $T_{2,1}$ is bounded. 
The proof that $T_{2,0}$ is bounded is similar.

Now let $\theta\in(0,1)$ and let us show that $T_{2,\theta}$ is bounded.
Consider the open strip 
$$
\S=\bigl\{z\in\Cdb\, :\,  0<{\rm Re}(z)<1\bigr\}.
$$ 
Let $x,a\in M_a$ and define $F\colon\overline{\S}\to\Cdb$ by
$$
F(z)= {\rm Tr}\Bigl(T
\Bigl(\sigma^{\varphi}_{\frac{i}{2}(1-z)}(x)\Bigr) D^\frac12
\sigma^{\varphi}_{-\frac{iz}{2}}(a) D^\frac12\Bigr).
$$
This is a well-defined function which is actually the restriction
to $\overline{\S}$ of an entire function. For all $t\in\Rdb$, we have
\begin{align*}
F(it) & = {\rm Tr}\Bigl(D^\frac12\,T
\Bigl(\sigma^{\varphi}_{\frac{i}{2}}\bigl(\sigma^{\varphi}_{\frac{t}{2}}(x)
\bigr)
\Bigr) D^\frac12
\sigma^{\varphi}_{\frac{t}{2}}(a)\Bigr)\\
& = 
{\rm Tr}\Bigl(D^\frac12\,T
\Bigl(D^{-\frac12}\sigma^{\varphi}_{\frac{t}{2}}(x) D^\frac12
\Bigr) D^\frac12
\sigma^{\varphi}_{\frac{t}{2}}(a) \Bigr)\\
&=
{\rm Tr}\Bigl(T_{2,0}
\Bigl(\sigma^{\varphi}_{\frac{t}{2}}(x) D^\frac12
\Bigr) 
D^\frac12 \sigma^{\varphi}_{\frac{t}{2}}(a)\Bigr),
\end{align*}
by (\ref{MAG2}). Hence by (\ref{MAG}),
\begin{align*}
\vert F(it)\vert &\leq \Bignorm{T_{2,0}
\Bigl(\sigma^{\varphi}_{\frac{t}{2}}(x)D^\frac12\Bigr)}_2
\Bignorm{D^\frac12 \sigma^{\varphi}_{\frac{t}{2}}(a)}_2\\
&\leq \bignorm{T_{2,0}} 
\bignorm{D^{\frac{it}{2}}(xD^\frac12)D^{-\frac{it}{2}}}_2
\bignorm{D^{\frac{it}{2}}(D^\frac12 a)D^{-\frac{it}{2}}}_2\\ 
&= \bignorm{T_{2,0}} \bignorm{xD^\frac12}_2
\bignorm{D^\frac12 a}_2.
\end{align*}
Likewise,
$$
F(1+it) = {\rm Tr}\Bigl(T_{2,1}
\Bigl(\sigma^{\varphi}_{\frac{t}{2}}(x) D^\frac12
\Bigr) 
D^\frac12 \sigma^{\varphi}_{\frac{t}{2}}(a)\Bigr),
$$
hence
$$
\vert F(1+it)\vert \leq \bignorm{T_{2,1}} \bignorm{xD^\frac12}_2
\bignorm{D^\frac12 a}_2.
$$
By the three lines lemma, we deduce that 
$$
\vert F(\theta)\vert \leq \bignorm{T_{2,0}}^{1-\theta}
\bignorm{T_{2,1}}^\theta\bignorm{xD^\frac12}_2
\bignorm{D^\frac12 a}_2.
$$
To calculate $F(\theta)$, we apply (\ref{MAG2}) again
and we obtain 
\begin{align*}
F(\theta) & = {\rm Tr}\Bigl(T
\bigl(D^{-\frac{1-\theta}{2}} x D^{\frac{1-\theta}{2}}\bigr) D^\frac12
D^{\frac{\theta}{2}} a  D^{-\frac{\theta}{2}}D^\frac12\Bigr)\\
& = {\rm Tr}\Bigl( D^{\frac{1-\theta}{2}}\,T
\bigl(D^{-\frac{1-\theta}{2}} x D^\frac12 
D^{-\frac{\theta}{2}}\bigr) 
D^{\frac{\theta}{2}} D^\frac12 a \Bigr)\\
& = {\rm Tr}\Bigl(T_{2,\theta}\bigl(xD^\frac12\bigr) D^\frac12 a\Bigr).
\end{align*}
Thus,
$$
\Bigl\vert {\rm Tr}\Bigl(T_{2,\theta}\bigl(xD^\frac12\bigr) D^\frac12 a\Bigr)
\Bigr\vert\leq  \bignorm{T_{2,0}}^{1-\theta}
\bignorm{T_{2,1}}^\theta\bignorm{xD^\frac12}_2
\bignorm{D^\frac12 a}_2.
$$
Since $M_aD^\frac12$ and $D^\frac12M_a$ are both dense 
in $L^2(M,\varphi)$, this estimate shows that $T_{2,\theta}$ is bounded, with
$\norm{T_{2,\theta}}\leq 
\bignorm{T_{2,0}}^{1-\theta}
\bignorm{T_{2,1}}^\theta$.

We now let $p\in(2,\infty)$. The proof in this 
case is a variant of the proof of \cite[Theorem 5.1]{HJX}. 
We use Kosaki's theorem which is presented after Lemma \ref{D-Gamma}, see
(\ref{Cp}) and (\ref{Ko1}).
Let $\theta\in[0,1]$. Let 
${\mathfrak J}_\theta\colon  M\to L^2(M,\varphi)$ be defined
by ${\mathfrak J}_\theta(x)=D^{\frac{1-\theta}{2}} xD^\frac{\theta}{2}$
for all $x\in M$. Equip  ${\mathfrak J}_\theta(M)$ with
\begin{equation}\label{NormM2}
\bignorm{D^{\frac{1-\theta}{2}} x 
D^\frac{\theta}{2}}_{{\mathfrak J}_\theta(M)} = \norm{x}_M,
\qquad x\in M.
\end{equation}
Consider $({\mathfrak J}_\theta (M),L^2(M,\varphi))$ as an 
interpolation couple. In analogy with (\ref{Cp}), we set
$$
E(p,\theta)=\bigl[{\mathfrak J}_\theta(M),L^2(M,\varphi)\bigr]_{\frac{2}{p}},
$$
subspace of $L^2(M,\varphi)$ given by the complex interpolation method.
Let $q\in(2,\infty)$ such that 
$$
\frac{1}{p}\,+\,\frac{1}{q}\,=\,\frac12.
$$
We introduce one more mapping
$U_\theta\colon L^2(M,\varphi)\to L^1(M,\varphi)$ defined by
$$
U_\theta(\zeta) = D^{\frac{1-\theta}{2}}\zeta D^{\frac{\theta}{2}},
\qquad \zeta\in L^2(M,\varphi).
$$
By (\ref{Ko1}), $U_\theta$ is an isometric isomorphism from 
$L^2(M,\varphi)$ onto $C(2,\theta)$. 
Since $U_\theta$ restricts to 
an isometric isomorphism from ${\mathfrak J}_\theta(M)$
onto $J_\theta(M)$, by (\ref{NormM1}) and (\ref{NormM2}),
it induces an isometric isomorphism
from $E(p,\theta)$ onto $\bigl[J_\theta(M),C(2,\theta)\bigr]_{\frac{2}{p}}$.
By (\ref{Cp}) and the reiteration theorem for complex interpolation
(see \cite[Theorem 4.6.1]{BL}), the latter is equal to $C(p,\theta)$.
Hence 
$U_\theta$ actually induces an isometric isomorphism
\begin{equation}\label{Iso}
E(p,\theta) \stackrel{U_\theta}{\simeq} 
C(p,\theta).
\end{equation}
Since $\frac{1}{p'}=\frac12+\frac{1}{q}\,$, we have
$$
U_\theta\bigl(D^{\frac{1-\theta}{q}} y D^{\frac{\theta}{q}}\bigr)
= D^{\frac{1-\theta}{p'}} y D^{\frac{\theta}{p'}}
$$
for all $y\in L^p(M,\varphi)$. Applying (\ref{Ko1}) and
(\ref{Iso}), we deduce that 
$$
E(p,\theta) = D^{\frac{1-\theta}{q}}L^{p}(M,\varphi)D^{\frac{\theta}{q}},
$$
with
\begin{equation}\label{NormEp}
\bignorm{D^{\frac{1-\theta}{q}}y
D^{\frac{\theta}{q}}}_{E(p,\theta)}
=\norm{y}_{L^p(M,\varphi)},\qquad y\in L^p(M,\varphi).
\end{equation}

Now let 
$$
S=T_{2,\theta}\colon L^2(M,\varphi)\longrightarrow L^2(M,\varphi)
$$
be given by the first part of the proof
(boundedness of $T_{2,\theta}$). By (\ref{NormM2}), 
$S$ is bounded on ${\mathfrak J}_\theta (M)$. Hence by the interpolation theorem, 
$S$ is bounded on $E(p,\theta)$.

Using (\ref{NormEp}), we deduce that for all $x\in M$,
\begin{align*}
\bignorm{D^{\frac{1-\theta}{p}} T(x)
D^{\frac{\theta}{p}}}_{L^p(M,\varphi)} 
& = \bignorm{D^{\frac{1-\theta}{2}} T(x)
D^{\frac{\theta}{2}}}_{E(p,\theta)}\\
& \leq \bignorm{S\colon E(p,\theta)\to E(p,\theta)} \bignorm{D^{\frac{1-\theta}{2}}x
D^{\frac{\theta}{2}}}_{E(p,\theta)}\\
& = \bignorm{S\colon E(p,\theta)\to E(p,\theta)}  \bignorm{D^{\frac{1-\theta}{p}} x
D^{\frac{\theta}{p}}}_{L^p(M,\varphi)}.
\end{align*}
This proves that $T_{p,\theta}$ is bounded and completes the proof.
\end{proof}

\begin{remark}\label{Estimate} Let $T\colon M\to M$ be a 
$2$-positive map such that $\varphi\circ T\leq C_1 T$ 
for some $C_1\geq 0$ and let $C_\infty=\norm{T}$.
It follows from the above proof and an obvious scaling that for any 
$p\geq 2$ and any $\theta\in[0,1]$, we have
$$
\bignorm{T_{p,\theta}\colon L^p(M,\varphi)\longrightarrow L^p(M,\varphi)}
\leq C_{\infty}^{1-\frac{1}{p}}C_1^{\frac{1}{p}}.
$$
\end{remark}

\begin{theorem}\label{GoodCase} 
Let $T\colon M\to M$ be a $2$-positive map
such that $\varphi\circ T\leq \varphi$ and let $1\leq p\leq 2$.
If 
\begin{equation}\label{Interval}
1 -\frac{p}{2}\,\leq \theta \leq \,\frac{p}{2},
\end{equation}
then $T_{p,\theta}\colon \A_{p,\theta}\to 
\A_{p,\theta}$ extends to a 
bounded map $L^p(M,\varphi)\to L^p(M,\varphi)$.
\end{theorem}

\begin{proof}
We will use Theorem \ref{p=2} on $L^2(M,\varphi)$, as well
as the fact that $T_{1,\frac12}$ is bounded, see \cite[Lemma 5.3]{HJX} or Remark \ref{Known}. 
Let $p\in(1,2)$, let $\theta$ satisfying (\ref{Interval}), and let 
$$
\eta= \frac{\theta -\bigl(1-\frac{p}{2}\bigr)}{p-1}.
$$
Then $\eta\in[0,1]$. This interpolation number is chosen in such a way that
\begin{equation}\label{eta}
\frac{\eta}{p'}\, +\, \frac{1-\theta}{p}\, =\, \frac{\theta}{p}\,+\, \frac{1-\eta}{p'}\,=\,\frac12\,,
\end{equation}
where $p'$ is the conjugate number of $p$.

We set 
$$
S=T_{1,\frac12}\colon L^1(M,\varphi)\longrightarrow L^1(M,\varphi).
$$
Let $V\colon L^2(M,\varphi)\to L^1(M,\varphi)$ defined by
$V(y)=D^{\frac{\eta}{2}} y D^{\frac{1-\eta}{2}}$  for all $y\in L^2(M,\varphi)$.
According to (\ref{Ko1}), $V$ is an isometric isomorphism from $L^2(M,\varphi)$ onto $C(2,1-\eta)$. 
Hence for all $x\in M$, we have
\begin{align*}
\bignorm{S(D^\frac12 x D^\frac12)}_{C(2,1-\eta)}
& = \bignorm{D^{\frac{\eta}{2}} D^{\frac{1-\eta}{2}}  T(x)  D^{\frac{\eta}{2}}  D^{\frac{1-\eta}{2}} }_{C(2,1-\eta)}\\
& = \bignorm{D^{\frac{1-\eta}{2}}  T(x)  D^{\frac{\eta}{2}}}_{L^2(M,\varphi)}\\
&\leq \bignorm{T_{2,\eta}}\bignorm{D^{\frac{1-\eta}{2}} x  D^{\frac{\eta}{2}}}_{L^2(M,\varphi)}\\
& = \bignorm{T_{2,\eta}}\bignorm{D^\frac12 
xD^\frac12 }_{C(2,1-\eta)}.
\end{align*}
Here the boundedness of $T_{2,\eta}$ is provided by Theorem  \ref{p=2}.
This proves that $S$ is bounded on $C(2,1-\eta)$.

By (\ref{Cp}) and the reiteration theorem, we have
$$
C(p,1-\eta)=\bigl[C(2,1-\eta),L^1(M,\varphi)\bigr]_{\frac{2}{p}-1}.
$$
Therefore, $S$ is bounded on $C(p,1-\eta)$. Using (\ref{Ko1}) again, as well as (\ref{eta}), we deduce that for any $x\in M$,
\begin{align*}
\bignorm{D^{\frac{1-\theta}{p}} T(x) D^{\frac{\theta}{p}}}_{L^p(M,\varphi)}
& = \bignorm{D^{\frac{\eta}{p'}} D^{\frac{1-\theta}{p}} T(x) D^{\frac{\theta}{p}} D^{\frac{1- \eta}{p'}}}_{C(p,1-\eta)}\\
&= \bignorm{D^\frac12 T(x) D^{\frac12}}_{C(p,1-\eta)}\\
& \leq \bignorm{S\colon C(p,1-\eta)\to C(p,1-\eta)}
\bignorm{D^\frac12 x D^{\frac12}}_{C(p,1-\eta)}\\
& = \bignorm{S\colon C(p,1-\eta)\to C(p,1-\eta)}
\bignorm{D^{\frac{1-\theta}{p}} x D^{\frac{\theta}{p}} }_{L^p(M,\varphi)}.
\end{align*}
This shows that $T_{p,\theta}$ is bounded.
\end{proof}

\section{The use of infinite tensor products}\label{TP}
In this section, we  show how to reduce the problem
of constructing a unital completely positive map $T\colon (M,\varphi)\to
(M,\varphi)$ such that $\varphi\circ T=\varphi$ and $T_{p,\theta}$ is unbounded,
for a certain pair $(p,\theta)$, to 
a finite dimensional question. In the sequel, by a matrix algebra $A$, we mean an algebra $A=M_n$ for some 
$n\geq 1$.

\begin{lemma}\label{FiniteProduct}
Let $A_1,A_2$ be two matrix algebras and for $i=1,2$,
consider a faithful state $\varphi_i$ on $A_i$.
Let $B=A_1\otimes_{\rm min} A_2$ and consider the faithful state
$\psi=\varphi_1\otimes \varphi_2$ on $B$. Let
$T_i\colon A_i\to A_i$ be a linear map, for $i=1,2$,
and consider $T=T_1\otimes T_2\colon B\to B$. Then
for any $1\leq p<\infty$ and any $\theta\in[0,1]$,
we have
\begin{align*}
\bignorm{T_{p,\theta}\colon L^p(B,\psi) 
&\to L^p(B,\psi)} \geq \\
&\bignorm{\{T_1\}_{p,\theta}\colon L^p(A_1,\varphi_1)\to 
L^p(A_1,\varphi_1)} \bignorm{\{T_2\}_{p,\theta}\colon L^p(A_2,\varphi_2)\to 
L^p(A_2,\varphi_2)}.
\end{align*}
\end{lemma}

\begin{proof}
Let $n_1,n_2\geq 1$ such that $A_1=M_{n_1}$ and $A_2=M_{n_2}$
and let $n=n_1n_2$.
For $i=1,2$, let $\Gamma_i\in M_{n_i}$ such that $\varphi_i(X_i)={\rm tr}(\Gamma_iX_i)$
for all $X_i\in M_{n_i}$. As in Proposition \ref{Transfer}, consider the mapping 
$\{U_i\}_{p,\theta}\colon S^p_{n_i}\to S^p_{n_i}$ defined by
$\{U_i\}_{p,\theta}(Y_i)=
\Gamma_i^{\frac{1-\theta}{p}} T_i\bigl(\Gamma_i^{-\frac{1-\theta}{p}}
Y_i\Gamma_i^{-\frac{\theta}{p}}\bigr)\Gamma_i^{\frac{\theta}{p}}$ for all
$Y_i\in S^p_{n_i}$. Using the standard identification 
\begin{equation}\label{Ident}
B=M_{n_1}\otimes_{\rm min} M_{n_2}\simeq M_n,
\end{equation}
we observe that $\psi(X) ={\rm tr}\bigl((\Gamma_1\otimes\Gamma_2)X)\bigr)$ for all $X\in M_n$. 
Hence using the identification $S^p_n=S^p_{n_1}\otimes S^p_{n_2}$ inherited
from (\ref{Ident}), we obtain the the mapping $U_{p,\theta}$ defined by
(\ref{Up}) is actually given by
$$
U_{p,\theta}=\{U_1\}_{p,\theta}\otimes \{U_2\}_{p,\theta}.
$$

For any $Y_1\in S^p_{n_1}$ and $Y_2\in S^p_{n_2}$, we have  $\norm{Y_1\otimes Y_2}_p = \norm{Y_1}_p \norm{Y_2}_p$. Hence
we deduce 
\begin{align*}
\norm{\{U_1\}_{p,\theta}(Y_1)}
\norm{\{U_2\}_{p,\theta}(Y_2)} & =\norm{\{U_1\}_{p,\theta}(Y_1)\otimes
\{U_2\}_{p,\theta}(Y_2)}\\ & =\norm{U_{p,\theta}(Y_1\otimes Y_2)}\\ & \leq \norm{U_{p,\theta}}
\norm{Y_1}_p\norm{Y_2}_p.
\end{align*}
This implies that $\norm{\{U_1\}_{p,\theta}}\norm{\{U_2\}_{p,\theta}}\leq\norm{U_{p,\theta}}$
Applying Proposition \ref{Transfer}, we obtain the requested inequality.
\end{proof}

Throughout the rest of this section, we let $(A_k)_{k\geq 1}$ 
be a sequence of matrix algebras.
For any $k\geq 1$, let $\varphi_k$ be a faithful state
on $A_k$. Let 
$$
(M,\varphi)=\overline{\otimes}_{k\geq 1}(A_k,\varphi_k)
$$
be the infinite tensor product associated with the 
$(A_k,\varphi_k).$ We refer to \cite[Section XIV.1]{T3}
for the construction and the properties of this tensor product. 
We merely recall that if we regard
$(A_1\otimes\cdots\otimes A_n)_{n\geq 1}$ as an increasing
sequence of (finite-dimensional) algebras
in the natural way, then
\begin{equation}\label{DenseB}
\B:=\bigcup_{n\geq 1} A_1\otimes\cdots\otimes A_n
\end{equation}
is $w^*$-dense in $M$. Further, $\varphi$ is a normal faithful state on $M$
such that
$$
\varphi_1 \otimes\cdots\otimes\varphi_n = 
\varphi_{\vert A_1\otimes\cdots
\otimes A_n},
$$
for all $n\geq 1$.

\begin{proposition}\label{Infinite}
Let $1\leq p<\infty$ and $\theta\in[0,1]$.
For any $k\geq 1$, let $T_k\colon A_k\to A_k$ be a unital completely 
positive map such that 
$\varphi_k\circ T_k=\varphi_k$. 
Assume that 
$$
\prod_{k=1}^n\bignorm{\{T_k\}_{p,\theta}\colon L^p(A_k,\varphi_k)\to L^p(A_k,\varphi_k)}
\longrightarrow\infty\qquad\hbox{when}\ n\to\infty.
$$
Then there exists a unital completely positive map $T\colon M\to M$ such that
$\varphi\circ T=\varphi$ and $T_{p,\theta}$ is unbounded.    
\end{proposition}

\begin{proof} For any $n\geq 1$, we introduce 
$B_n=A_1\otimes_{\rm min}\cdots\otimes_{\rm min} A_n$ and the faithful state
$$
\psi_n=\varphi_1 \otimes\cdots\otimes\varphi_n 
$$ 
on $B_n$.
According to \cite[Proposition XIV.1.11]{T3}, the modular automorphism
group of $\varphi$ preserves $B_n$. Consequently (see Remark \ref{CondExp}),
there  exists a unique normal conditional expectation $E_n \colon M\to B_n$
such that $\varphi=\psi_n\circ E_n$, and the pre-adjoint of $E_n$ yields an
isometric embedding
$$
L^1(B_n,\psi_n)\hookrightarrow L^1(M,\varphi).
$$
Likewise, let $F_n\colon B_{n+1}\to B_n$ be the conditional
expectation defined by 
\begin{equation}\label{Fn}
F_n(a_1\otimes\cdots\otimes a_n\otimes a_{n+1}) =
\varphi_{n+1}(a_{n+1}) \,a_1\otimes\cdots\otimes a_n,
\end{equation}
for all $a_1\in A_1,\ldots, a_n\in A_n, a_{n+1}\in A_{n+1}$. Then the
pre-adjoint of $F_n$ yields an isometric embedding
$$
J_n\colon L^1(B_n,\psi_n)\hookrightarrow L^1(B_{n+1},\psi_{n+1}).
$$
We can therefore consider $\bigl(L^1(B_n,\psi_n)\bigl)_{n\geq 1}$
as an increasing sequence of subspaces of $L^1(M,\varphi)$.
We introduce 
$$
\mathcal L :=\bigcup_{n\geq 1} L^1(B_n,\psi_n)\,\subset\, L^1(M,\varphi).
$$
Let $D\in L^1(M,\varphi)$ be the density of $\varphi$.
It follows from Remark \ref{CondExp} that 
$$
\mathcal L = D^\frac12 \B D^\frac12,
$$
where $\B$ is defined by (\ref{DenseB}). 
Since $\B$ is $w^*$-dense, it is dense in $M$ for the strong 
operator topology given by the standard representation $M\hookrightarrow B(L^2(M,\varphi))$.
Hence by
\cite[Lemma 2.2]{J}, $\B D^\frac12$ is dense in 
$L^2(M,\varphi)$. This implies that 
$\mathcal L$ is dense in $L^1(M,\varphi)$. 

For any $n\geq 1$, let
$$
V(n) := T_1\otimes\cdots\otimes T_n\colon B_n\longrightarrow B_n.
$$
This is a unital completely positive map. 
Hence its norm is equal to 1.
Let 
$$
S_n=V(n)_*\colon
L^1(B_n,\psi_n)\longrightarrow L^1(B_n,\psi_n)
$$
be the pre-adjoint of $V(n)$. Then $\norm{S_n}=1$.
We observe that 
\begin{equation}\label{Comm}
J_n\circ S_n = S_{n+1}\circ J_n.
\end{equation}
Indeed by duality, this amounts to show that $V(n)\circ F_n=F_n\circ V(n+1)$,
where $F_n$ is given by (\ref{Fn}). The latter is 
true because $\varphi_{n+1}\circ T_{n+1}=\varphi_{n+1}$.

Thanks to (\ref{Comm}), one may define
$$
\S\colon \mathcal L \longrightarrow \mathcal L
$$
by letting $\S(\eta)=S_n(\eta)$ if $\eta\in L^1(B_n,\psi_n)$.
Then $\S$ is bounded, with $\norm{\S}=1$. 
Owing to the density of $\mathcal L$, there exists a unique bounded  
$S\colon L^1(M,\varphi)\to L^1(M,\varphi)$ extending 
$\S$. Using the duality (\ref{1Dual-L1}), we set 
$$
T=S^*\colon M \longrightarrow M.
$$
By construction, $T$ is a contraction. Furthermore,
for each $n\geq 1$, $S_n^*=V(n)$
is a unital completely positive map and $\psi_{n}\circ S_n^*=\psi_n$.
We deduce that 
$T$ is unital and completely positive and that 
$$
\varphi\circ T=\varphi.
$$

Let $1\leq p<\infty$ and let $\theta\in[0,1]$.
Let us use the isometric embedding 
\begin{equation}\label{5Emb}
L^p(B_n,\psi_n)\hookrightarrow L^p(M,\varphi)
\end{equation}
as explained in Remark \ref{CondExp}. If $D_n$ denotes the density of $\psi_n$, then it follows from
\cite[Proposition 5.5]{HJX} that the embedding (\ref{5Emb}) 
maps $D_n^{\frac{1-\theta}{p}} x D_n^{\frac{\theta}{p}}$
to $D^{\frac{1-\theta}{p}} x D^{\frac{\theta}{p}}$ for all
$x\in B_n$.
Then the restriction of $T_{p,\theta}\colon \A_{p,\theta}\to L^p(M,\varphi)$ coincides
with 
$$
V(n)_{p,\theta}\colon L^p(B_n,\psi_n)\longrightarrow L^p(B_n,\psi_n).
$$
Finally we observe that by a simple iteration of Lemma 
\ref{FiniteProduct}, we have
$$
\norm{V(n)_{p,\theta}}\geq 
\prod_{k=1}^n\bignorm{\{T_k\}_{p,\theta}\colon L^p(A_k,\varphi_k)\to L^p(A_k,\varphi_k)}.
$$
The assumption that this product of norms tends 
to $\infty$ therefore implies that  the operator 
$T_{p,\theta}$ is unbounded.
\end{proof}

\section{Non-extension results}\label{No} 
The aim of this section is to show the following.

\begin{theorem}\label{Main}
Let $1\leq p<2$.  If either
\begin{equation}\label{Choice}
0\leq \theta<\frac12\bigl(1-\sqrt{p-1}\bigr)
\qquad\hbox{or}\qquad \frac12\bigl(1+\sqrt{p-1}\bigr)<\theta\leq 1,
\end{equation}
then there exist a von Neumann algebra $M$ equipped with a normal faithful state
$\varphi$, as well as a unital completely positive map $T\colon M\to M$ 
such that $\varphi\circ T=\varphi$ and the mapping 
$T_{p,\theta}\colon \A_{p,\theta}\to \A_{p,\theta}$ defined by (\ref{T}) is unbounded.
\end{theorem}

This result will be proved at the end of this section, as a simple combination of 
Proposition \ref{Infinite} and the following key result. Recall that $M_2$ denotes the
space of $2\times 2$ matrices.

\begin{proposition}\label{Key}
Let $1\leq p<2$ and let $\theta\in[0,1]$ be satisfying (\ref{Choice}).
Then there exist a unital completely positive map $T\colon M_2\to M_2$ and
a faithful state $\varphi$ on $M_2$ such that $\varphi\circ T=\varphi$
and $\norm{T_{p,\theta}}>1$.
\end{proposition}

\begin{proof}
Let $c\in(0,1)$ and consider 
$$
\Gamma =\begin{pmatrix} 1-c & 0 \\ 0 & c\end{pmatrix}.
$$
This is a positive invertible matrix with trace equal to $1$. We let 
$\varphi$ denote its associated
faithful state on $M_2$, that is, 
$\varphi(X)={\rm tr}(\Gamma X)=(1-c)x_{11} 
+cx_{22}$, for all $X=\begin{pmatrix}  x_{11} & x_{12} \\ x_{21} 
& x_{22}\end{pmatrix}$
in $M_2$.

Let $E_{i,j}$, $1\leq i,j\leq 2$, denote the standard matrix units of $M_2$.
Let $T\colon M_2\to M_2$ be the linear map defined by
$$
T(E_{11})=(1-c)I_2,\quad T(E_{22})=cI_2,\quad \hbox{and}\quad
T(E_{21})=T(E_{12})= \bigl(c(1-c)\bigr)^\frac12\bigl(E_{12}+E_{21}\bigr).
$$
Let $A=\bigl[T(E_{ij})\bigr]_{1\leq i,j\leq 2}\in M_2(M_2)$. If we regard $A$ as an element of $M_4$, we have
$$
A=\begin{pmatrix} 1-c & 0 & 0 &  (c(1-c))^\frac12 \\ 0 &1-c &  (c(1-c))^\frac12 & 0\\
0 & (c(1-c))^\frac12 &  c & 0 \\ (c(1-c))^\frac12 & 0 & 0 & c\end{pmatrix}.
$$
Clearly $A$ is unitarily equivalent to $B\otimes I_2$, with
$$
B= \begin{pmatrix} 1-c & (c(1-c))^\frac12 \\  (c(1-c))^\frac12 & c\end{pmatrix}.
$$
It is plain that $B$ is positive. Consequently, $A$ is positive. 
Hence $T$ is completely positive, by Choi's theorem (see e.g. 
\cite[Theorem 3.14]{Pa}). Furthermore, $T$ is unital.
We note that $\varphi(T(E_{11}))= \varphi(E_{11})=1-c$,
$\varphi(T(E_{22}))= \varphi(E_{22})=c$, $\varphi(T(E_{12}))= \varphi(E_{12})=0$
and $\varphi(T(E_{21}))= \varphi(E_{21})=0$. Thus,
$$
\varphi\circ T=\varphi.
$$

Our aim is now to estimate $\norm{T_{p,\theta}}$, using 
Proposition \ref{Transfer}. We let  
$U_{p,\theta}\colon S^p_2\to S^p_2$ be defined by
(\ref{Up}). We shall focus on the action of 
$U_{p,\theta}$ on the anti-diagonal part of $S^p_2$.
First, we have
$$
\Gamma^{-\frac{1-\theta}{p}}E_{12}\Gamma^{-\frac{\theta}{p}}
= (1-c)^{-\frac{1-\theta}{p}}  c^{-\frac{\theta}{p}} E_{12}.
$$
Hence
\begin{align*}
T\bigl(\Gamma^{-\frac{1-\theta}{p}}E_{12}\Gamma^{-\frac{\theta}{p}}\bigr) &
= (1-c)^{-\frac{1-\theta}{p}}  c^{-\frac{\theta}{p}} T(E_{12})\\ &
= (1-c)^{-\frac{1-\theta}{p}}  c^{-\frac{\theta}{p}}
(c(1-c))^\frac12\bigl(E_{12}+E_{21}\bigr).
\end{align*}
Hence
\begin{align*}
U_{p,\theta}(E_{12}) &
= (1-c)^{-\frac{1-\theta}{p}}  c^{-\frac{\theta}{p}}
(c(1-c))^\frac12\Bigl(\Gamma^{\frac{1-\theta}{p}}E_{12} \Gamma^{\frac{\theta}{p}}
+\Gamma^{\frac{1-\theta}{p}}E_{21} \Gamma^{\frac{\theta}{p}}\Bigr)\\ &
= (1-c)^{-\frac{1-\theta}{p}}  c^{-\frac{\theta}{p}}
(c(1-c))^\frac12\Bigl( (1-c)^{\frac{1-\theta}{p}}  c^{\frac{\theta}{p}} E_{12}
+ c^{\frac{1-\theta}{p}}  (1-c)^{\frac{\theta}{p}} E_{21}\Bigr)\\ &
= (c(1-c))^\frac12\biggl( E_{12} + \Bigl(\frac{1-c}{c}\Bigr)^{\frac{2\theta-1}{p}}E_{21}\biggr).
\end{align*}
Likewise, we have
$$
U_{p,\theta}(E_{21}) = (c(1-c))^\frac12\biggl(\Bigl(\frac{c}{1-c}\Bigr)^{\frac{2\theta-1}{p}}E_{12} + E_{21}\biggr).
$$

Set 
\begin{equation}\label{delta}
\delta=\Bigl(\frac{1-c}{c}\Bigr)^{\frac{2\theta-1}{p}}.
\end{equation}
Consider 
$$
Y=\begin{pmatrix} 0 & a \\ b & 0\end{pmatrix}
\qquad\hbox{with}\quad \vert a\vert^p+\vert b\vert^p=1,
$$
so that $\norm{Y}_p=1$. Then
\begin{align*}
U_{p,\theta}(Y) & = (c(1-c))^\frac12\bigl(aE_{12} 
+a\delta E_{21} + b\delta^{-1} E_{12} + b E_{21}\bigr)\\
& = (c(1-c))^\frac12\bigl((a+b\delta^{-1}) E_{12} + (a\delta +b)E_{21}\bigr).
\end{align*}
Hence
\begin{equation}\label{Y}
\norm{U_{p,\theta}(Y)}_p^p= (c(1-c))^\frac{p}{2}\bigl(
(a+b\delta^{-1})^p +(a\delta +b)^p\bigr).
\end{equation}

To prove 
Proposition \ref{Key},
it therefore suffices to show that for any
$1\leq p<2$ and  $\theta\in[0,1]$ satisfying (\ref{Choice}), there exist
$a,b>0$ and $c\in(0,1)$ such that 
$$
a^p+b^p=1\qquad\hbox{and}\qquad 
(c(1-c))^\frac{p}{2}\bigl(
(a+b\delta^{-1})^p +(a\delta +b)^p\bigr)>1,
$$
where $\delta$ is given by (\ref{delta}).

We first assume that \underline{$p>1$.} 
We let $q=\frac{p}{p-1}$ denote its conjugate
exponent. 
Given $c\in(0,1)$ and $\delta$ as above, we define
\begin{equation}\label{ab}
a=\biggl(\frac{\delta^q}{1+\delta^{q}}\biggr)^{\frac{1}{p}}
\qquad\hbox{and}\qquad
b=\biggl(\frac{1}{1+\delta^{q}}\biggr)^{\frac{1}{p}}.
\end{equation}
They satisfy $a^p+b^p=1$ as required. Note that these values of $(a,b)$ are chosen
in order to maximize the quantity
$(c(1-c))^\frac{p}{2}\bigl(
(a+b\delta^{-1})^p +(a\delta +b)^p\bigr)$,
 according to 
the Lagrange multiplier method.

We set 
$$
c_t=\frac12+t,\qquad -\frac12 <t<\frac12.
$$
Then we denote by $\delta_t,a_t,b_t$ the real numbers 
$\delta,a,b$ defined by (\ref{delta}) and (\ref{ab})
when $c=c_t$. Also we set 
$$
\gamma_t=(c_t(1-c_t))^\frac{p}{2}\qquad \hbox{and}\qquad
{\mathfrak m}_t= \gamma_t\bigl(
(a_t+b_t\delta_t^{-1})^p +(a_t\delta_t +b_t)^p\bigr).
$$
It follows from above that
it suffices to show that ${\mathfrak m}_t>1$ 
for some $t\in\bigl(0,\frac12\bigr)$.
We will prove this property by writing the second order
Taylor expansion of ${\mathfrak m}_t$.

We have
$$
(a_t+b_t\delta_t^{-1})^p +(a_t\delta_t +b_t)^p  = 
(1+\delta_t^{-p})(a_t\delta_t +b_t)^p.
$$
Moreover
$$
a_t\delta_t = \frac{\delta_t^{\frac{q}{p} +1}}{(1+\delta_t^q)^{\frac{1}{p}}}\,
=\frac{\delta_t^{q}}{(1+\delta_t^q)^{\frac{1}{p}}}\,.
$$
Hence 
$$
(a_t+b_t\delta_t^{-1})^p +(a_t\delta_t +b_t)^p = 
(1+\delta_t^{-p}) (1+ \delta_t^q\bigr)^{p-1}.
$$
Consequently, 
$$
{\mathfrak m}_t = 
\gamma_t
(1+\delta_t^{-p}) (\delta_t^q +1 \bigr)^{p-1}.
$$

In the sequel, we write 
$$
f_t\equiv g_t
$$
when $f_t=g_t+o(t^2)$ when $t\to 0$.

We note that $c_t(1-c_t)=\bigl(\frac12+t\bigr)\bigl(\frac12-t\bigr)
=\frac14\bigl(1-4t^2\bigr)$. We deduce that
\begin{equation}\label{gt}
\gamma_t\equiv  \frac{1}{2^p} (1-2pt^2).
\end{equation}

We set $\lambda=2\theta-1$ for convenience. Then we have
\begin{align*}
\delta_t & =\Bigl(\frac{1-2t}{1+2t}\Bigr)^{\frac{\lambda}{p}} \\
&\equiv \bigl((1-2t)(1-2t+4t^2)\bigr)^{\frac{\lambda}{p}} \\
&\equiv (1-4t +8t^2)^{\frac{\lambda}{p}} \\
&\equiv 1 -\frac{4\lambda}{p} t + \frac{8\lambda}{p}t^2 +
\frac12\,\frac{\lambda}{p}\Bigl(\frac{\lambda}{p}-1\Bigr)(4t)^2\\
&\equiv 1 -\frac{4\lambda}{p} t + \frac{8\lambda^2}{p^2}t^2.
\end{align*}
This implies that
\begin{align*}
\delta_t^q & \equiv 1 -\frac{4\lambda q}{p} t 
+ \frac{8\lambda^2 q}{p^2}t^2
+\frac12\,q(q-1)\Bigl(\frac{4\lambda}{p}\Bigr)^2 t^2\\
&\equiv 1 -\frac{4\lambda q}{p} t + \frac{8\lambda^2 q^2}{p^2}t^2.
\end{align*}
Likewise,
\begin{equation}\label{-p}
\delta_t^{-p} \equiv 1 +4\lambda t + 8\lambda^2t^2.
\end{equation}
Since $p-1=\frac{p}{q}$, we have
\begin{align*}
(1+\delta_t^q)^{p-1}
& \equiv 2^{\frac{p}{q}}\Bigl(
1 -\frac{2\lambda q}{p} t + \frac{4\lambda^2 q^2}{p^2} t^2\Bigr)^{\frac{p}{q}}\\
& \equiv 2^{\frac{p}{q}}\Bigl(1 -2\lambda t+\frac{4\lambda^2 q}{p} t^2
+\frac12\,\frac{p}{q}\Bigl(\frac{p}{q}-1\Bigr)
\Bigl(\frac{2\lambda q}{p}\Bigr)^2 t^2\Bigr)\\
&\equiv 2^{\frac{p}{q}}\bigl(1-2\lambda t + 2 \lambda^2 qt^2\bigr).
\end{align*}
Combining this expansion with (\ref{gt}) and (\ref{-p}), we deduce that
\begin{align*}
{\mathfrak m}_t &\equiv \frac{1}{2^p} (1-2pt^2)\,\cdotp
2(1 +2\lambda t + 4\lambda^2t^2)
\,\cdotp 2^{\frac{p}{q}}(1-2\lambda t + 2 \lambda^2 qt^2)\\
&\equiv (1-2pt^2)(1+2\lambda^2 qt^2).
\end{align*}
Consequently,
\begin{equation}\label{Alpha}
{\mathfrak m}_t \equiv 1+\alpha t^2\qquad\hbox{with}\qquad \alpha = 2(\lambda^2q-p).
\end{equation}
The second order coefficient $\alpha$ can be written as
\begin{align*}
\alpha & = 2q \Bigl((2\theta -1)^2 -\frac{p}{q}\Bigr)\\
& =8q\Bigl(\theta^2-\theta +\frac{q-p}{4q}\Bigr)\\
& =8q(\theta-\theta_0)(\theta-\theta_1),
\end{align*}
with
$$
\theta_0=\frac12\bigl(1-\sqrt{p-1}\bigr)
\qquad\hbox{and}\qquad
\theta_1=\frac12\bigl(1+\sqrt{p-1}\bigr).
$$

Now assume (\ref{Choice}). Then $\alpha>0$. Hence (\ref{Alpha}) 
ensures the existence
of $t>0$ such that ${\mathfrak m}_t>1$, which  
concludes the proof (in the case $p>1$).

We now consider the case \underline{$p=1$.} We apply the same method as before,
with
$$
a=1\qquad \hbox{and}\qquad b=0. 
$$
According to (\ref{Y}), 
it will suffice  to show that whenever $\theta\not=\frac12$, there exists 
$c\in(0,1)$ such that 
$(c(1-c))^\frac12(1+\delta)>1$.

Again we set $c_t=\frac12+t$, for $-\frac12<t<\frac12$,
we define $\delta_t$ accordingly and we set
$$
{\mathfrak m}_t = (c_t(1-c_t))^\frac12(1+\delta_t).
$$
It follows from the previous calculations that 
$$
(c_t(1-c_t))^\frac12 =\frac12 +o(t)\qquad\hbox{and}\qquad  
\delta_t = 1 -4(2\theta-1) t +o(t). 
$$
Consequently
$$
{\mathfrak m}_t =1 -2(2\theta-1) t +o(t).
$$
This order one expansion ensures
that if $\theta\not=\frac12$, then there exists
$t\in\bigl(-\frac12,\frac12\bigr)$ such that 
${\mathfrak m}(t)>1$, which  
concludes the proof (in the case $p=1$). 
\end{proof}

\begin{proof}[Proof of Theorem \ref{Main}]
Let $(p,\theta)$ satisfying (\ref{Choice}). Thanks to 
Proposition \ref{Key}, let $T_0\colon M_2\to M_2$ and let
$\varphi_0$ be a faithful state on $M_2$ such that $\varphi_0\circ T_0=\varphi_0$
and $\norm{\{T_0\}_{p,\theta}}>1$. We apply Proposition \ref{Infinite} with
$(A_k,\varphi_k,T_k)=(M_2,\varphi_0,T_0)$ for all $k\geq 1$.
In this case, 
$$
\prod_{k=1}^n\bignorm{\{T_k\}_{p,\theta}} = \norm{\{T_0\}_{p,\theta}}^n,
$$
and the latter goes to $\infty$ when $n\to\infty$. Hence $T_{p,\theta}$ is unbounded.
\end{proof}

\begin{remark}
With Theorem \ref{p=2}, Theorem \ref{GoodCase} and Theorem \ref{Main}, 
we have solved Question \ref{Q} in the following cases: 
(i) $p\geq 2$ and $\theta\in[0,1]$; 
(ii) $1\leq p<2$ and $\theta\in\bigl[1-p/2,p/2\bigr]$; 
(iii) $1\leq p<2$ and 
$\theta\in\bigl[0,2^{-1}(1-\sqrt{p-1})\bigr)$; 
(iv)  $1\leq p<2$ and 
$\theta\in \bigl(2^{-1}(1+\sqrt{p-1}),1\Bigr]$.

However we do not know the answer to Question \ref{Q} when 
$1\leq p<2$ and 
$$
\theta\in\bigl[2^{-1}(1-\sqrt{p-1}), 1-p/2\bigr)
\qquad\hbox{or}\qquad
\theta\in\bigl(p/2, 2^{-1}(1+\sqrt{p-1})\bigr].
$$
Writing a $(+)$ when Question 2.2 has a positive answer, a $(-)$ when it has a negative answer and a $(?)$ when we do not know the answer, we obtain the following diagram:
\bigskip

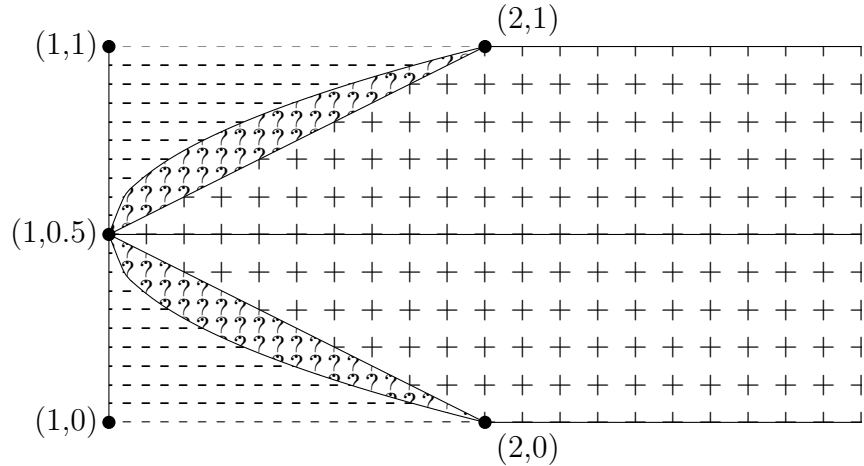
\begin{center}
\begin{tikzpicture}[scale=5]
  \draw[domain=1:3,smooth,variable=\x] plot ({\x},{0.5}) node[right]{};

  \draw[domain=2:3,smooth,variable=\x] plot ({\x},{1}) node[right]{};

  \draw[domain=2:3,smooth,variable=\x] plot ({\x},{0}) node[right]{};
  
  \draw[] (1,0) -- (1,1) node[above] {};
  
  \draw[domain=1:2,smooth,variable=\x] plot ({\x},{\x/2}) {};
  
  \draw[domain=1:2,smooth,variable=\x] plot ({\x},{1-(\x/2)}) {};
  
  \draw[domain=1:2,smooth,variable=\x] plot ({\x},{0.5*(1+sqrt(\x-1))}) node[right] {};
  
  \draw[domain=1:2,smooth,variable=\x] plot ({\x},{0.5*(1-sqrt(\x-1))}) node[right] {};
  
  \begin{scope}
    \clip (1,0.5) -- (2,0) -- (2,1) -- (1,0.5) -- cycle;
    \foreach \i in {0,...,30}
      \foreach \j in {0,...,30}
        \node at ({1+\i/10},{\j/10}) {+};
  \end{scope}
    \begin{scope}
    \clip (2,1) -- (3,1) -- (3,0) -- (2,0) -- cycle;
    \foreach \i in {0,...,30}
      \foreach \j in {0,...,30}
        \node at ({2+\i/10},{\j/10}) {+};
  \end{scope}

  \begin{scope}
    \clip (1,0.5) -- plot[smooth,domain=1:2] ({\x},{0.5*(1-sqrt(\x-1))}) -- (2,0) -- (1,0.5) -- cycle;
    \foreach \i in {0.05,0.1,...,2}
      \foreach \j in {0.05,0.1,...,1}
        \node at (\i,\j) {?};
  \end{scope}
    \begin{scope}
    \clip (1,0.5) -- plot[smooth,domain=1:2] ({\x},{0.5*(1+sqrt(\x-1))}) -- (1,0.5) -- (2,0) -- cycle;
    \foreach \i in {0.05,0.1,...,2}
      \foreach \j in {0.05,0.1,...,1}
        \node at (\i,\j) {?};
  \end{scope}

 
  \begin{scope}
   \clip (1,0.5) -- plot[smooth,domain=1:2] ({\x},{0.5*(1-sqrt(\x-1))})  -- (1,0.5) -- (2,0) -- (1,0) -- (1,0.5) -- cycle;
   \foreach \i in {0,0.05,0.1,...,2.1}
      \foreach \j in {0,0.05,0.1,...,1.1}
        \node at (\i,\j) {-};
  \end{scope}
    \begin{scope}
   \clip (1,0.5) -- plot[smooth,domain=1:2] ({\x},{0.5*(1+sqrt(\x-1))})  -- (1,0.5) -- (2,1) -- (1,1) -- (1,0.5) -- cycle;
   \foreach \i in {0.05,0.1,...,2.1}
      \foreach \j in {0,0.05,0.1,...,1.1}
        \node at (\i,\j) {-};
  \end{scope}
\fill (1,0.5) circle (0.5pt) node[left] {(1,0.5)};
\fill (1,1) circle (0.5pt) node[left] {(1,1)};
\fill (1,0) circle (0.5pt) node[left] {(1,0)};
\fill (2,1) circle (0.5pt) node[above right] {(2,1)};
\fill (2,0) circle (0.5pt) node[below right] {(2,0)};
\end{tikzpicture}
\end{center}
\end{remark}

\bigskip
\noindent
{\bf Acknowledgements.} 
We thank \'Eric Ricard for several stimulating discussions.
The first author was supported by the ANR project Noncommutative
analysis on groups and quantum groups (No./ANR-19-CE40-0002).
The authors gratefully thank the support of the Heilbronn Institute for Mathematical Research and the UKRI/EPSRC Additional Funding Program.

\bigskip

\begin{thebibliography}{99}
\bibitem{A} C. Arhancet, {\it On Matsaev's conjecture for contractions 
on noncommutative $L_p$-spaces}, 
J. Operator Theory 69 (2013), no. 2, 387–421.
\bibitem{A2} Arhancet, {\it Dilations of semigroups on von Neumann 
algebras and noncommutative $L_p$-spaces},
J. Funct. Anal. 276 (2019), no. 7, 2279–2314.
\bibitem{BL} J. Bergh  and J. L\"ofstr\"om,
{\it Interpolation spaces, an introduction},  
Grundlehren der Mathematischen Wissenschaften, 
No. 223. Springer-Verlag, Berlin-New York, 1976, x+207 pp.
\bibitem{CDS} M. Caspers and M. de la Salle,
{\it Schur and Fourier multipliers of an amenable group acting on non-commutative $L_p$-spaces},  Trans. Amer. Math. Soc. 367 (2015), no. 10, 6997–7013. 
\bibitem{Choi} M.-D. Choi, 
{\it A Schwarz inequality for positive linear maps 
on $C^*$-algebras}, Illinois J. Math. 18 (1974), 565–574.
\bibitem{DDDP} P. Dodds, T. Dodds and B. de Pagter, 
{\it Noncommutative Banach function spaces},
Math. Z. 201 (1989), no. 4, 583–597. 
\bibitem{DL} C. Duquet and C. Le Merdy, {\it  A characterization of absolutely dilatable Schur multipliers}, Adv. Math. 439 (2024), Paper No. 109492.
\bibitem{H} U. Haagerup, {\it $L^p$-spaces associated with an abritrary 
von Neumann algebra},  In ``Algèbres d'opérateurs et leurs applications en physique mathématique (Proc. Colloq., Marseille, 1977)", pp. 175–184, Colloq. Internat. CNRS, 274, CNRS, Paris, 1979. 
\bibitem{HJX} U. Haagerup, M. Junge and Q. Xu, 
{\it A reduction method for noncommutative $L_p$-spaces and applications},
Trans. Amer. Math. Soc. 362 (2010), no. 4, 2125–2165. 
\bibitem{Hiai} F. Hiai, {\it Lectures on selected topics in
von Neumann algebras}, EMS Series of Lectures in Mathematics, 
EMS Press, Berlin, 2021.
\bibitem{HRW} G. Hong, S. K. Ray and S. Wang, {\it 
Maximal ergodic inequalities for some positive operators on noncommutative $L_p$-spaces}, J. Lond. Math. Soc. (2) 108 (2023), no. 1, 362–408. 
\bibitem{J} M. Junge, {\it Doob's inequality for non-commutative martingales},
J. Reine Angew. Math. 549 (2002), 149–190.
\bibitem{JL} M. 
Junge and N. LaRacuente, {\it 
Multivariate trace inequalities, $p$-fidelity, and universal recovery beyond tracial settings}, 
J. Math. Phys. 63 (2022), no. 12, Paper No. 122204, 43 pp.
\bibitem{JLX} M. Junge, C. Le Merdy and Q. Xu, {\it 
$H^\infty$-functional calculus and square functions on noncommutative 
$L_p$-spaces} Astérisque Soc. Math. France 305 (2006), vi+138 pp.
\bibitem{JX1} M. Junge and Q. Xu, {\it Noncommutative Burkholder/Rosenthal inequalities}, Annals of Prob. 31 (2003), 948-995.
\bibitem{JX}  M. Junge and Q. Xu, {\it Noncommutative maximal ergodic theorems},
J. Amer. Math. Soc. 20 (2007), no. 2, 385–439.
\bibitem{Ko} H. Kosaki, 
{\it Applications of the complex interpolation method to a von Neumann algebra: noncommutative $L_p$-spaces}, 
J. Funct. Anal. 56 (1984), 29–78.
\bibitem{Pa} V. Paulsen, {\it Completely bounded maps and operator algebras},
Cambridge Studies in Advanced Mathematics, 78,
Cambridge University Press, Cambridge, 2002, xii+300 pp. 
\bibitem{PX} G. Pisier and Q. Xu, {\it Non-commutative $L^p$-spaces}, 
Handbook of the geometry of Banach spaces, Vol. 2, 1459–1517, North-Holland, Amsterdam, 2003.
\bibitem{T2}  M. Takesaki, {\it Theory of operator algebras II},
Encyclopaedia of Mathematical Sciences, 125,
Springer-Verlag, Berlin, 2003.
\bibitem{T3}  M. Takesaki, {\it Theory of operator algebras III},
Encyclopaedia of Mathematical Sciences, 127,
Springer-Verlag, Berlin, 2003. 
\bibitem{Terp} M. Terp, {\it $L_p$-spaces associated with von Neumann algebras}, Notes, Math. Institute, Copenhagen University, 1981.
\end{thebibliography}
\end{document}